\documentclass[a4paper,12pt]{amsart}
\usepackage{amsbsy}
\usepackage{amssymb}
\usepackage{graphicx,epsfig,subfigure,psfrag}
\usepackage{color}
\usepackage{graphicx}

\newtheorem{theorem}{Theorem}
\newtheorem{proposition}{Proposition}
\newtheorem{question}{Question}
\newtheorem{remark}{Remark}

\numberwithin{equation}{section} \numberwithin{theorem}{section}
\numberwithin{proposition}{section} \numberwithin{lemma}{section}
\numberwithin{corollary}{section}
\numberwithin{definition}{section} \numberwithin{remark}{section}

\newcommand{\R}{\mathbb{R}}

\author{Antoine Lemenant and Hayk Mikayelyan}

\title{Stationarity of the crack-front for the Mumford-Shah problem in 3D}
\begin{document}
\begin{abstract}  In this paper we exhibit a family of stationary solutions of the Mumford-Shah functional in $\R^3$, arbitrary close to a crack-front. Unlike other  examples, known in the literature, those are topologically non-minimizing in the sense of Bonnet \cite{b}. 

We also give a local version in a finite cylinder and prove an energy estimate for minimizers. Numerical illustrations   indicate the stationary solutions are unlikely minimizers and show how the dependence on axial variable impacts the geometry of the discontinuity set.  

A self-contained proof of the stationarity of the crack-tip function for the Mumford-Shah problem in 2D is presented. 
\end{abstract}

\maketitle

\tableofcontents

\thispagestyle{empty}

\section{Introduction and main statements} 

The Mumford-Shah functional 
$$J(u,K):=\int_{\Omega \backslash K}|\nabla u|^2 \; + \alpha (u-g)^2    \; d  x +\beta \mathcal{H}^{N-1}(K),$$
where $g\in L^\infty(\Omega)$ and   $u \in H^1(\Omega \setminus K)$, was introduced in the context of image processing  by Mumford and  Shah in \cite{ms}. The idea is to find, by minimizing the functional $J(u,K)$ among couples $(u,K)$, a ``piecewise smooth'' approximation of the given   image data $g$, together with the location of its edges, represented by the set $K$. Actually, the concept of competing bulk and surface energies is much older and goes back to   Griffith (see \cite{griff}), whose theory of brittle fracture is based on the balance between gain in surface energy and strain energy release. This idea translated to  the Mumford-Shah functional was exploited in the now classical variational model for crack propagation  by Francfort and Marigo \cite{fracf}. We refer to \cite{afp,fusco,d,focardi,lemRev} for overview references about the Mumford-Shah functional.

In this article we will take $\alpha=0$, $\beta=1$, and will consider local minimizers defined as follows.  
Let $\Omega \subset \R^N$ be open and let  $(u,K)$ be a couple such that  $K\subset \Omega$ closed and $u\in H^1(B \setminus K)$, for all balls  $B\subset \Omega$. We consider the local Mumford-Shah  energy in the ball $B$ given by
$$J(u,K,B):=\int_{B}|\nabla u|^2 \; dx +\mathcal{H}^{N-1}(K\cap B).$$
We are interested in couples $(u,K)$  that are locally minimizing in $\Omega$, i.e. such that 
$$J(u,K,B) \leq J(u',K',B)$$
for any $B\subset \Omega$ and for any competitor $(u',K')$ which satisfies $u=u'$ in $\Omega \setminus B$ and $K=K'$ in $\Omega \setminus B$. In that case we will simply say that $(u,K)$ is a Mumford-Shah minimizer.

As minimizers are known to be equivalently defined on the $SBV$ space (see \cite{afp}),  it is not restrictive to assume  $K$ being an  $(N-1)$-rectifiable set.

It is then classical that any such minimizer will satisfy the Euler-Lagrange equation associated to the Mumford-Shah functional, which is given for any $\eta\in C^1_c(\Omega)^N$ by the equation (see \cite[Theorem 7.35]{afp})
\begin{eqnarray}
\int_{\Omega}|\nabla u|^2 {\rm div} \,\eta-2\langle \nabla u, \nabla u \cdot \nabla \eta\rangle \;dx + \int_{K}{\rm div}^{K}\, \eta \;d\mathcal{H}^{N-1} =0. \label{equation}
\end{eqnarray}
A couple $(u,K)$ satisfying the equation \eqref{equation} for all $\eta \in C^1_c(\Omega)^N$ will be called \emph{stationary} in $\Omega$.

One of the most famous example of Mumford-Shah minimizer in $\R^2$ is   the so-called cracktip function, which is known to be the only non-constant element in the list of global minimizers of \cite{b}. This function plays also a fundamental role in fracture theory. Namely, we define 
$$K_0:=]-\infty,0]\times\{0\}\subset \R^2$$
 and 
$$\varphi_0(r,\theta):=\sqrt{\frac{2 r}{\pi}}\sin(\theta/2) \quad r>0, \quad \theta\in ]-\pi,\pi[. $$
It has been proved in a famous 250-pages paper by Bonnet and David \cite{Bonnetd} that the couple $(\varphi_0,K_0)$ is a Mumford-Shah minimizer in  $\R^2$. 

Let us notice that the crack-tip is the only singularity where the two competing terms of the Mumford-Shah energy, the Dirichlet energy and the surface area of the discontinuity set, scale of the same order. In almost all other singular points the surface area dominates and makes the successful application of the methods originating from the minimal surface theory possible.

In contrast, the cracktip singularity is  much more delicate and offers  many challenging  mathematical questions that are  also of great interest regarding to brittle fracture theory and crack propagation.

While dimension 2 is already well studied, the case of dimension 3 is even more open.  Indeed, it is possible to construct a 3D version of the cracktip function in $\R^3$  taking the vertically constant function defined by
$$u_0(r,\theta,z):=\varphi_0(r,\theta),$$
which admits as singular set the half-plane
$$P_0:=K_0\times \R.$$
By this way one obtains a couple $(u_0,P_0)$, the {\it crack-front}, which is a Mumford-Shah minimizer in $\R^3$ (see \cite{d}). Actually, one can do the same construction in any dimensions but for simplicity in this paper we restrict ourselves to dimensions 2 and 3 only.

A fortciori, as a byproduct of \cite{Bonnetd}, we know that  the cracktip function and its 3D analogue  both satisfy the Euler-Lagrange equation \eqref{equation}.  We shall give here a direct proof of the latter, without using \cite{Bonnetd}.

Our main interest in this paper is to get some understanding about the dependence of the function $u$ on the points on the front of the crack in 3D (see \cite{and-myk} for a recent preprint about regularity issues at the cracktip in 2D). We observe the following curious fact: in dimension 3, there exists a family of stationary couples $(u_\delta,P_0)$, which are not minimizing, and which are arbitrary close to the crack-front function.  More precisely we define
$$u_\delta(r,\theta,z):=\varphi_0(r,\theta) + \delta z,$$
for some parameter $\delta  \in \R$.  The associated singular set is still the half-plane $P_0$. We then prove the following.

\begin{theorem}  \label{mainth} The couple $(u_\delta,P_0)$ is stationary in $\R^3$ for any $\delta \in \R$, but is minimizing  in $\R^3$ if and only if $\delta=0$.
\end{theorem}

It is worth mentioning that constructing an example of non-minimizing stationary couple is rather easy. For instance take a line $L\subset \R^2$ as singular set and two constants on each sides as function, that we denote by $u_L$. It is easy to see that this couple $(u_L,L)$ is stationary, but non minimizing in $\R^2$. On the other hand it is a topological minimizer in the sense of Bonnet \cite{b}.

More precisely, the way one usually proves that a line $L$ is non minimizing in $\R^2$, is to cut the line, i.e. considering a competitor of the form $L \setminus B(0,R)$, for some $R$ large enough (and find a convenient associated competitor $v$ for the function $u_L$ as well, by use of a cutoff function). 

This competitor is not admissible in the class of topological competitors of Bonnet \cite{b}, since the connected components of $\R^2 \setminus L$ has been changed. This is why, $(u_L,L)$ remains to be a topological minimizer while it is not a simple minimizer. 

According to our knowledge, all the non-minimizing stationary couples that are known so far are of the same type.

As a matter of fact, unlike the example of $(u_L,L)$, the stationary example $(u_\delta,P_0)$ is even not a topological minimizer since, in the case of the cracktip,   $\R^3\setminus P_0$ is connected and consequently,  the topological condition is trivial.   Anecdotally, this would mean that   $(u_\delta,P_0)$ might be the first known example of stationary non topological-minimizer.

Next, let us restrict  $u_\delta$ on the boundary of  the cylinder 
$$\mathcal{C}:=B_{2D}(0,1) \times [-1,1],$$
and consider the problem
\begin{eqnarray}
\min \big\{ \int_{\mathcal{C}}|\nabla u|^2 \;dx + \mathcal{H}^2(S_u) \; ; \;  u \in SBV(\mathcal{C})  \text{ and } u=u_\delta \text{ on } \partial \mathcal{C}\big\}. \label{problem0}
\end{eqnarray}

It is clear from the $SBV$ theory  that the above problem admits a minimizer, and it  is  very natural to ask the following 

\begin{question}\label{question1} Is $u_\delta$ a solution of Problem \eqref{problem0} for some $\delta \not = 0$ ?
\end{question}

This question can be motivated by a related result by Mora and Morini   \cite{moriniLocal}, which says in some cases  that  ``stationarity implies minimality''. More precisely, the result in \cite{moriniLocal} is the following: take a smooth curve $K$ which cuts the unit  ball of $\R^2$ in two parts and assume that $u$ is a function such that  $(u,K)$ is stationary in the ball. Then there exists a neighborhood of $K$ in which $(u,K)$ is minimizing. The proof uses some calibration technics.

The statement applies for instance to the simple example mentioned earlier of a line in $\R^2$, with two constants on each sides as function $u$. This couple is stationary in  $\R^2$, it is not minimizing in $\R^2$, but it is minimizing if one restrict it on a bounded and thin enough cylinder around $K$ (the thickeness depending on the ``jump'', i.e.  the difference of the two constants).

It is then tempting to consider the same situation for the function $u_\delta$. It is even more intriguing thinking that no one has even proved the existence of a calibration for the cracktip so far, which means that an analogue of Mora and Morini's result for a cracktip situation seems  nowadays out of reach by use of similar technics. Subsequently, it might not be possible to  answer to Question~\ref{question1} by use of a calibration.

On the other hand it is easy to exclude a too large $\delta$, namely we prove the following.

\begin{proposition} \label{prop1} If $\delta > \sqrt{3-\frac{2}{\pi}} $ then $u_\delta$ is not a minimizer of problem \eqref{problem0}.
\end{proposition}

At the opposite, concerning small $\delta>0$, we have the following energy estimate.

\begin{proposition} \label{propsitionsdelta}Let $s(\delta):=J(u_\delta,K_\delta)-\min J$. Then $s(\delta)=o(\delta^2)$ as $\delta\to 0$.
\end{proposition}

In a last section we give some numerical results showing what would be the shape of a minimizer for problem \eqref{problem0}.


\section{Stationarity for the cracktip in 2D}
\label{section2}
As we said before the stationarity of the cracktip is a direct consequence of \cite{Bonnetd} but we shall give here a direct and independent proof. 

\begin{proposition} The cracktip   $(\varphi_0,K_0)$ is stationary.  
\end{proposition}

\begin{proof}   Let $\eta \in C^1_c(\Omega)^2$. Then integrating by parts on $K_0$ gives 
$$\int_{K_0}{\rm div}^{K_0}\, \eta \, d\mathcal{H}^{1}= \eta(0) \cdot {\bf e}_1 .$$
Next concerning the other terms in \eqref{equation}, as the integration by parts in the fractured domain $\R^2\setminus K_0$ is not clear,  we consider the Lipschitz domain $\Omega_\varepsilon:=\R^2\setminus B(0,\varepsilon)$ and we notice that 
$$\lim_{\varepsilon \to 0}\int_{B_\varepsilon}|\nabla \varphi_0|^2 {\rm div} \,\eta-2\langle \nabla \varphi_0, \nabla \varphi_0 \cdot \nabla \eta\rangle \;dx =0.$$

Consequently, to prove the Euler-Lagrange equation  \eqref{equation}  it is enough to prove 
\begin{eqnarray}
\lim_{\varepsilon \to 0}\int_{\Omega_\varepsilon}|\nabla \varphi_0|^2 {\rm div} \,\eta-2\langle \nabla \varphi_0, \nabla \varphi_0 \cdot \nabla \eta\rangle \;dx =-\eta(0)\cdot {\bf e}_1.
\end{eqnarray}
For this purpose we integrate by parts in the Lipschitz domain $\Omega_\varepsilon$  yielding 
$$\int_{\Omega_\varepsilon}|\nabla \varphi_0|^2 {\rm div}\, \eta=-2\int_{\Omega_\varepsilon}\langle \eta, \nabla^2 \varphi_0\cdot \nabla \varphi_0 \rangle dx - \int_{K_0\cap \partial \Omega_\varepsilon}[|\nabla \varphi_0|^2]^\pm \langle\eta , \nu \rangle - \int_{\partial B_\varepsilon}|\nabla \varphi_0|^2\langle\eta , \nu \rangle$$
where  $\nu$ in the last integral denotes the inner normal vector, and  $[|\nabla \varphi_0|^2]^\pm$ is the jump of $|\nabla \varphi_0|^2$ on $K_0$. Actually it is easy to check from the very definition of $\varphi_0$ that $[|\nabla \varphi_0|^2]^\pm=0$ on $K_0$ (the details will be given  just below) thus
\begin{eqnarray}
\int_{\Omega_\varepsilon}|\nabla \varphi_0|^2 {\rm div}\, \eta=-2\int_{\Omega_\varepsilon}\langle \eta, \nabla^2 \varphi_0\cdot \nabla \varphi_0 \rangle dx  - \int_{\partial B_\varepsilon}|\nabla \varphi_0|^2\langle\eta , \nu \rangle. \label{firstterm}
\end{eqnarray}
Indeed, to check that $[|\nabla \varphi_0|^2]^\pm=0$ we compute $|\nabla \varphi_0|^2$ in polar coordinates. Let us denote $C_0:=\sqrt{\frac{2}{\pi}}$ so that $\varphi_0=C_0\sqrt{r}\sin(\theta/2)$. We get
\begin{eqnarray}
\frac{\partial}{\partial r} \varphi_0(r,\theta)=  C_0\frac{1}{2\sqrt{r}}\sin(\theta/2). \label{grad1}
\end{eqnarray}
\begin{eqnarray}
\frac{\partial }{\partial \theta} \varphi_0(r,\theta)= C_0\sqrt{r}\frac{1}{2}\cos(\theta/2) \label{grad2}
\end{eqnarray}
so that 
$$|\nabla \varphi_0|^2=\left| \frac{\partial }{\partial r} \varphi_0\right|^2 + \left| \frac{\partial }{r\partial \theta} \varphi_0\right|^2=\frac{C_0^2}{4r}.$$
In particular it does not depend on $\theta$,   the claim is now proven and so as \eqref{firstterm}.

Next, using that $\varphi_0$ has zero normal derivative on $K_0$, and is harmonic in $\Omega_\varepsilon$,
\begin{eqnarray}
2\int_{\Omega_\varepsilon} \langle \nabla \varphi_0 , \nabla \varphi_0 \cdot \nabla \eta \rangle
&=&-2\int_{\Omega_\varepsilon}[\Delta \varphi_0\langle \eta,\nabla \varphi_0 \rangle+\langle \eta, \nabla^2\varphi_0\cdot \nabla \varphi_0 \rangle]dx +2\int_{\partial B(0,\varepsilon)} \frac{\partial \varphi_0}{\partial \nu} \langle \eta , \nabla \varphi_0\rangle \notag \\
&=& -2\int_{\Omega_\varepsilon}\langle \eta, \nabla^2\varphi_0\cdot \nabla \varphi_0 \rangle dx +2\int_{\partial B(0,\varepsilon)} \frac{\partial \varphi_0}{\partial \nu} \langle \eta , \nabla \varphi_0\rangle. \notag
\end{eqnarray}
Then subtracting with \eqref{firstterm} we finally get
$$\int_{\Omega_\varepsilon}|\nabla \varphi_0|^2 {\rm div}\, \eta-2\int_{\Omega_\varepsilon} \langle \nabla \varphi_0 , \nabla \varphi_0 \cdot \nabla \eta \rangle= - \int_{\partial B(0,\varepsilon)}|\nabla \varphi_0|^2\langle\eta , \nu \rangle-2\int_{\partial B(0,\varepsilon)} \frac{\partial \varphi_0}{\partial \nu} \langle \eta , \nabla \varphi_0\rangle.$$
We want to take the limit as $\varepsilon \to 0$. We treat the two terms separately, beginning with the first one, 

\begin{eqnarray}
-\int_{\partial B_\varepsilon}|\nabla \varphi_0|^2\langle\eta , \nu \rangle =-\int_{-\pi}^\pi \frac{C_0^2}{4\varepsilon}\langle\eta , \nu \rangle \varepsilon d\theta &=& \frac{C_0^2}{4} \int_{-\pi}^\pi \langle \eta(\varepsilon,\theta) ,  (\cos(\theta),\sin(\theta))\rangle d\theta \notag \\
&=&   \frac{C_0^2}{4} \int_{-\pi}^\pi   \eta_1(\varepsilon,\theta)\cos(\theta)+\eta_2(\varepsilon,\theta)\sin(\theta) d\theta. \notag
\end{eqnarray}
Now as $\varepsilon \to 0$ we get 
\begin{eqnarray}
  \frac{C_0^2}{4} \int_{-\pi}^\pi   \eta_1(\varepsilon,\theta)\cos(\theta)+\eta_2(\varepsilon,\theta)\sin(\theta) d\theta &\to_{\varepsilon \to 0}& \frac{C_0^2}{4} \int_{-\pi}^\pi   \eta_1(0)\cos(\theta)+\eta_2(0)\sin(\theta) d\theta \notag \\
  &=& 0. \notag
  \end{eqnarray}

Next we compute the contribution of the second term,  denoting by $\tau$ and $-\nu$ the tangential and normal unit vectors on the circle $\partial B(0,\varepsilon)$,
\begin{eqnarray}
-2\int_{\partial B(0,\varepsilon)} \frac{\partial \varphi_0}{\partial \nu} \langle \eta , \nabla \varphi_0\rangle&=& 2\int_{-\pi}^\pi \frac{\partial \varphi_0}{\partial r} \langle \eta , \nabla \varphi_0\rangle\; \varepsilon \; d\theta \notag \\
&=& 2\int_{-\pi}^\pi \frac{\partial \varphi_0}{\partial r} \left( \langle \eta, \tau \rangle  \frac{\partial \varphi_0}{\varepsilon \partial \theta} +  \langle \eta, \nu \rangle \frac{\partial \varphi_0}{\partial r} \right)\varepsilon \;d\theta .\label{III}
\end{eqnarray}
Then we compute  using \eqref{grad1} and \eqref{grad2},
\begin{eqnarray}
 \frac{\partial \varphi_0}{\partial r}  \langle \eta, \tau \rangle  \frac{\partial \varphi_0}{\varepsilon \partial \theta}&=&\frac{C_0^2}{4\varepsilon} \sin(\theta/2)\cos(\theta/2)\langle(\eta_1,\eta_2),(-\sin(\theta),\cos(\theta))\rangle \notag \\
 &=&  -\frac{C_0^2}{4\varepsilon} \eta_1 \sin(\theta) \sin(\theta/2)\cos(\theta/2) +\frac{C_0^2}{4\varepsilon} \eta_2 \cos(\theta) \sin(\theta/2)\cos(\theta/2) \label{A1} 
 \end{eqnarray}
and 
\begin{eqnarray}
 \frac{\partial \varphi_0}{\partial r}   \langle \eta, \nu \rangle \frac{\partial \varphi_0}{\partial r}
 &=&  \left(\frac{\partial \varphi_0}{\partial r}\right)^2\langle(\eta_1,\eta_2) , (\cos(\theta),\sin(\theta)) \rangle \notag \\
 &=& \frac{C_0^2}{4\varepsilon}\eta_1 \sin^2(\theta/2)\cos(\theta) +  \frac{C_0^2}{4\varepsilon} \eta_2\sin^2(\theta/2) \sin(\theta). \label{A2} 
\end{eqnarray}

Next we use 
$$\cos(\theta/2)=\cos(\theta-\theta/2)=\cos(\theta)\cos(\theta/2)+\sin(\theta)\sin(\theta/2)$$
$$\sin(\theta/2)=\sin(\theta-\theta/2)=\sin(\theta)\cos(\theta/2)-\cos(\theta)\sin(\theta/2),$$
so that adding \eqref{A1} and \eqref{A2} together yields

$$ \frac{\partial \varphi_0}{\partial r} \langle \eta , \nabla \varphi_0\rangle=-\frac{C_0^2}{4\varepsilon} \eta_1\sin^2(\theta/2)+\frac{C_0^2}{4\varepsilon} \eta_2\sin(\theta/2)\cos(\theta/2).$$
Returning to \eqref{III} we infer that 
\begin{eqnarray}
\lim_{\varepsilon\to 0} -2\int_{\partial B(0,\varepsilon)} \frac{\partial \varphi_0}{\partial \nu} \langle \eta , \nabla \varphi_0\rangle&=&-\frac{C_0^2}{2}\int_{-\pi}^\pi \eta_1(0)\sin^2(\theta/2)+\frac{C_0^2}{2} \eta_2(0)\sin(\theta/2)\cos(\theta/2)\; d\theta \notag \\
&=& -\frac{1}{\pi}\eta_1(0)\int_{-\pi}^\pi  \sin^2(\theta/2) \;d\theta .\notag 
\end{eqnarray}
 Then from the identity $\sin^2(\alpha)=\frac{1-\cos(2\alpha)}{2}$ we deduce that $\int_{-\pi}^\pi  \sin^2(\theta/2) \;d\theta = \pi$ thus
 $$\lim_{\varepsilon\to 0} -2\int_{\partial B(0,\varepsilon)} \frac{\partial \varphi_0}{\partial \nu} \langle \eta , \nabla \varphi_0\rangle = -\eta_1(0)=-\eta(0)\cdot {\bf e}_1$$
 and this finishes the proof of the proposition.
\end{proof}

Using a similar proof in 3D we can prove the following. The details are left to the reader. 

\begin{proposition} The 3D-cracktip   $(u_0,P_0)$ is stationary.  
\end{proposition}

\section{Stationarity for $u_\delta$}
\label{section3}

We can now verify the stationarity of $u_\delta$.

\begin{proposition}\label{udelta} The couple  $(u_\delta,P_0)$ is stationary for any $\delta\in \R$.
\end{proposition}
\begin{proof} Let $\eta\in C^1_c(\R^3)$. We already know that $u_0$ is stationary, which means 
\begin{eqnarray}
\int_{\R^3}|\nabla u_0|^2 {\rm div} \,\eta-2\langle \nabla u_0, \nabla u_0 \cdot \nabla \eta\rangle \;dx + \int_{P_0}{\rm div}^{P_0}\, \eta \;d\mathcal{H}^{N-1} =0. \label{equation1}
\end{eqnarray}
Let us now check the same equality with $u_\delta$ in place of $u_0$.

Since $u_\delta=u_0+\delta z$ we have
$$\int_{\R^3}|\nabla u_\delta|^2 {\rm div} \,\eta=\int_{\Omega}|\nabla u_0|^2 {\rm div} \,\eta + \int_{\Omega} \delta^2 {\rm div} \,\eta = \int_{\R^3}|\nabla u_0|^2 {\rm div} \,\eta $$
because $\eta$ has compact support thus   $\int_{\R^3}{\rm div}\; \eta =0$. 
In addition, $\nabla u_\delta=\nabla u_0 + \delta {\bf e}_3$ thus

$$\langle \nabla u_\delta, \nabla u_\delta \cdot \nabla \eta\rangle =\langle \nabla u_0, \nabla u_0 \cdot \nabla \eta\rangle+
\langle\delta {\rm e}_3, \nabla u_0 \cdot \nabla \eta\rangle+
\langle \nabla u_0, \delta {\rm e}_3 \cdot \nabla \eta\rangle+
\delta^2\langle {\rm e}_3, {\rm e}_3 \cdot \nabla \eta\rangle.$$
To prove the stationarity of  $u_\delta$, it is enough to prove
\begin{eqnarray}
\int_{\R^3} \langle\delta {\rm e}_3, \nabla u_0 \cdot \nabla \eta\rangle+
\langle \nabla u_0, \delta {\rm e}_3 \cdot \nabla \eta\rangle+
\delta^2\langle {\rm e}_3, {\rm e}_3 \cdot \nabla \eta\rangle dx =0. \label{isZero}
\end{eqnarray}
Let us treat each term separately. We write  by $\eta^k$ for $k=1,2,3$ the components of $\eta$. The first term  turns out to be
$$\int_{\R^3} \langle\delta {\rm e}_3, \nabla u_0 \cdot \nabla \eta\rangle =
\delta \int_{\R^3} \partial_x u_0 \,\partial_z \eta^1 +\partial_y u_0 \,\partial_z \eta^2  .$$
Now the latter is equal to zero because $\eta$ has compact support and $\partial_xu_0$ and $\partial_yu_0$ do not depend on the $z$ variable. Therefore by use of Fubini we can write
$$\delta \int_{\R^3} \partial_x u_0 \,\partial_z \eta^1 +\partial_y u_0 \,\partial_z \eta^2=\delta\int_{\R^2}\partial_x u_0\left(\int_{\R} \partial_z \eta^1 dz \right)dxdy+\delta\int_{\R^2}\partial_y u_0\left(\int_{\R} \partial_z \eta^2 dz \right)dxdy$$
which is equal to zero due to the fact that $\int_{\R} \partial_z \eta^k dz=0$ for $k=1,2$.
\medskip

The second term in \eqref{isZero} is 
$$\int_{\R^3}\langle \nabla u_0, \delta {\rm e}_3 \cdot \nabla \eta\rangle=\delta\int_{\R^3}\langle \nabla u_0 , \nabla \eta^3 \rangle= \delta\int_{\R}\left(\int_{\R^2} \langle \nabla_{x,y} \varphi_0, \nabla_{x,y}\eta^3\rangle dxdy\right) dz=0,$$
because of the Euler equation for the 2D cracktip.

\medskip

Finally we get rid of the third term as follows.

$$\int_{\R^3}\delta^2\langle {\rm e}_3, {\rm e}_3 \cdot \nabla \eta\rangle = \delta^2\int_{\R^3} \partial_z\eta^3 = \delta^2\int_{\R^2} \left(\int_{\R} \partial_z\eta^3 dz\right)dxdy=0.$$
All in all, we just have proved that $u_\delta$ is stationary. 
\end{proof}
\vspace{1cm}
Let us now give a proof of Theorem \ref{mainth}.

\begin{proof}[Proof of Theorem \ref{mainth}] We already know from Proposition \ref{udelta} that $(u_\delta,P_0)$ is stationary, and we know by \cite{Bonnetd} (see also \cite{d}) that it is minimizing for $\delta=0$. Let us prove that it is not minimizing for $\delta\not = 0$. Actually this is a direct consequence of \cite[Proposition 19]{l4}  but let us give here a simple and independent argument, using   the very standard competitor made by taking, for some $R>0$,
$$K':=(P_0\setminus B(0,R) )\cup \partial B(0,R),$$
and 
$$v=u_\delta {\bf 1}_{\R^3\setminus B(0,R)}.$$
The minimality  of $u_\delta$ would imply 
$$\int_{B(0,R)} |\nabla u_\delta|^2 \,dx + \mathcal{H}^2(P_0\cap B(0,R))\leq  \mathcal{H}^2( \partial B(0,R)),$$
and in particular
$$\delta^2 CR^3 \leq C'R^2,$$
which is a contradiction for $R$ large enough.
\end{proof}

\begin{remark} \label{rem-top} 
The competitor constructed in the proof above is admissible in the class of topological minimizers in the sense of Bonnet \cite{b}, since the topological condition   involves the separability of points in $\R^3\setminus B(0,R)$ only, and this separation has not been changed by our competitor (they all lie in the same connected component, as for the original set $P_0$).

However, we can even get a competitor $K'$ such that  $\R^3\setminus K'$ is connected   by ``drilling'' a very small hole on the sphere  $\partial B(0,R)$.

To see this let us take the point $y=(R,0,0)\in \partial B(0,R)$ and consider $ S_\varepsilon=\partial B(0,R) \setminus B(y,\varepsilon) $. 
Let us now take the competitor set $$K'_\varepsilon:=(P_0\setminus B(0,R) )\cup S_\varepsilon.$$
As competitor function,   we take $v_\varepsilon=u_\delta$ in $\R^3\setminus B(0,R)$, and
$$v_\epsilon=u_\delta \varphi_{\varepsilon}$$
inside $B(0,R)$, where  $\varphi_\varepsilon$ is a cutoff function defined as
\begin{equation*}
\begin{cases} \varphi_\varepsilon = 0  & \mbox{in } \R^3\setminus B(y,2{\varepsilon}),
 \\
  v_\varepsilon=1  & \mbox{on } B(y,\varepsilon),
  \end{cases}
\end{equation*}
$|v_\varepsilon|\leq 1$ and $|\nabla v_\varepsilon| \leq  \frac{C}{ \varepsilon}$.
Now we claim that 
$$
\int_{B(0,R)}|\nabla v_\varepsilon|^2\;dx\to 0 \text{ as }{\varepsilon\to 0},
$$
and thus, for the same $R$ as above, and for $\varepsilon$ small enough,
$$J(K'_\varepsilon,v_\varepsilon,B(0,R))< J(P_0,u_\delta,B(0,R)).$$

To prove the claim is suffice to estimate
$$|\nabla v_\varepsilon|^2\leq C( \varphi_\varepsilon^2 |\nabla u_\delta|^2 + u_\delta^2 |\nabla \varphi_\varepsilon|^2)$$
so that
\begin{eqnarray}
\int_{B(0,R)}|\nabla v_\varepsilon|^2&\leq& C \int_{B(y,2{\varepsilon})}|\nabla u_\delta|^2 \; dx + C\|u_\delta\|_{L^\infty(B(0,R))}^2\int_{B(y,2{\varepsilon})} \frac{1}{\varepsilon^2}  \notag \\
&\leq &  C \int_{B(y,2{\varepsilon})}|\nabla u_\delta|^2 \; dx + C' \varepsilon  ,\notag 
\end{eqnarray}
which goes to $0$ when $\varepsilon\to 0$.
\end{remark}


\section{Study in a cylinder}
\label{section4}
We now focus on Problem \ref{problem0}.

\begin{remark} \label{lem1} Let $u_0$ be the crack-front function in 3D. Then $\int_{\mathcal{C}}|\nabla u_0|^2 \; dx =2$.  Indeed, it suffice to compute $\int_{\mathcal{C}}|\nabla u_0|^2 \; dx$ in cylindrical coordinates  together with the fact that  $\int_{B(0,1)}|\nabla \varphi_0(x,y)|^2 dxdy =1$, where $\varphi_0$ is the 2D-cracktip. 
\end{remark}

\begin{proof}[Proof of Proposition \ref{prop1}]  The proof of Proposition \ref{prop1} follows again from a variant of the very standard competitor already used in the proof of Theorem \ref{mainth}. For any $\varepsilon>0$ small, we define
$$C_\varepsilon := B_{2D}(0,1-\varepsilon)\times [-1+\varepsilon, 1-\varepsilon] \subset \mathcal{C}$$
$$K_\varepsilon:= (K_0\setminus C_\varepsilon)  \cup \partial C_\varepsilon,$$
$$v= u_\delta {\bf 1}_{\mathcal{C}\setminus C_\varepsilon}.$$

Testing the local minimality of $(u_\delta,K_0)$ in $\mathcal{C}$ with this competitor $(v,K_\varepsilon)$ yields
$$\int_{\mathcal{C}}|\nabla u_\delta|^2 \; dx + \mathcal{H}^{2}(K_0\cap \mathcal{C})\leq \int_{\mathcal{C}\setminus C_\varepsilon}|\nabla u_\delta|^2 \; dx + \mathcal{H}^{2}(K_0\setminus C_\varepsilon)+\mathcal{H}^2(\partial C_\varepsilon).$$
Next, using Remark \ref{lem1} and using that 
$$\int_{\mathcal{C}}|\nabla u_\delta|^2 \; dx = \int_{\mathcal{C}}|\nabla u_0|^2 \; dx +\delta^2|\mathcal{C}|,$$
we get 
$$4+\delta^2|\mathcal{C}| \leq \int_{\mathcal{C}\setminus C_\varepsilon}|\nabla u_\delta|^2 \; dx + \mathcal{H}^{2}(K_0\setminus C_\varepsilon)+\mathcal{H}^2(\partial C_\varepsilon).$$
Letting $\varepsilon \to 0$ finally leads to
$$\delta^2\leq \frac{2\pi+4\pi-4}{2\pi}=3-\frac{2}{\pi},$$
and this finishes the proof of the Proposition.
\end{proof}

We now come to Proposition \ref{propsitionsdelta}. In the rest of this section we will denote by $J(u,K)$ for $J(u,K,\mathcal{C})$, i.e.
$$J(u,K):=\int_{\mathcal{C}\setminus K}|\nabla u|^2 dx +\mathcal{H}^2(K \cap \mathcal{C}).$$
 For any $\delta>0$ we denote by $(\bar u_\delta, \bar K_\delta)$ a minimizer for Problem  \eqref{problem0} and we recall that
$$ s(\delta):=J(u_\delta,K_\delta)-J(\bar u_\delta, \bar K_\delta)\geq 0.$$

In this section we prove in particular that  $s(\delta)=o(\delta^2)$ as $\delta\to 0$. We first observe the following easy bound
\begin{eqnarray}
 s(\delta)\leq 2\delta^2|\mathcal{C}|-2\delta \int_{\mathcal{C}\setminus \bar K_\delta} \frac{\partial \bar u_\delta}{\partial z} dx . \label{ineq00}
\end{eqnarray}

This follows by  simply noticing that $(\bar u_\delta -\delta z,\bar K_\delta)$ is a competitor for $(u_0,P_0)$, which is a minimizer, thus
$$J(u_0,P_0)\leq J(\bar u_\delta -\delta z,\bar K_\delta).$$
On the other hand
\begin{eqnarray}
\int_{\mathcal{C}\setminus \bar K_\delta}|\nabla (\bar u_\delta -\delta z)|^2 dx &=& \int_{\mathcal{C}\setminus \bar K_\delta}|\nabla \bar u_\delta|^2 dx-2\int_{\mathcal{C}\setminus \bar K_\delta}\langle  \nabla \bar u_\delta, \delta {\bf e}_3\rangle dx +\delta^2|\mathcal{C}|, \notag \\
&=& \int_{\mathcal{C}\setminus \bar K_\delta}|\nabla \bar u_\delta|^2 dx-2\delta \int_{\mathcal{C}\setminus \bar K_\delta} \frac{\partial \bar u_\delta}{\partial z} dx +\delta^2|\mathcal{C}|, \notag 
\end{eqnarray}
thus
\begin{eqnarray}
J(u_0,P_0)&\leq& J(\bar u_\delta ,\bar K_\delta)-2\delta \int_{\mathcal{C}\setminus \bar K_\delta} \frac{\partial \bar u_\delta}{\partial z} dx +\delta^2|\mathcal{C}| \notag \\
&= &  J( u_\delta ,P_0)-s(\delta)-2\delta \int_{\mathcal{C}\setminus \bar K_\delta} \frac{\partial \bar u_\delta}{\partial z} dx +\delta^2|\mathcal{C}| \notag \\
&= &  J( u_0 ,P_0)+\delta^2|\mathcal{C}|-s(\delta)-2\delta \int_{\mathcal{C}\setminus \bar K_\delta} \frac{\partial \bar u_\delta}{\partial z} dx +\delta^2|\mathcal{C}| ,\notag
\end{eqnarray}
so in conclusion we have proved \eqref{ineq00}. Notice that the right hand side of \eqref{ineq00} would be equal to zero if $\bar u_\delta=u_{\delta}$.

In what follows we establish a better inequality, as stated in the following proposition. 

\begin{proposition}\label{yetProp}
  
  \begin{eqnarray}
  s(\delta)\leq   \int_{\mathcal{C}} \left( \delta^2 - \left|\frac{\partial \bar u_\delta}{\partial z}\right| ^2  \right) dx+  \int_{\bar K_\delta} \left(|\sin(\theta(x))|-1 \right) \;d\mathcal{H}^2(x), \label{goodbound}
 \end{eqnarray} 
  where $\theta(x)$ is the angle between the normal vector $\nu$ to the tangent plane $T_x$ to $\bar K_\delta$ at point $x$, and the vector ${\bf e}_3$. 
\end{proposition}

\begin{remark} Notice again that the right-hand side of \eqref{goodbound} would be zero for $\bar u_\delta = u_\delta$.
\end{remark}

 \begin{proof}[Proof of Proposition \ref{yetProp}] By minimality of the 2D-cracktip function, we deduce that for any $z_0\in [-1,1]$, fixed, the function $(x,y)\mapsto u_\delta(x,y,z_0)$ is minimizing the 2D-Mumford-Shah energy with its own boundary datum on $\partial B(0,1)$. 

Therefore, since $(\bar u_\delta, \bar K_\delta)$ is a competitor, we obtain the inequality
$$\int_{B_{2D}(0,1)}|\nabla_{x,y} u_\delta|^2 dxdy+ \mathcal{H}^1(P_0\cap \{z=z_0\}) \leq \int_{B_{2D}(0,1)} |\nabla_{x,y} \bar u_{\delta}(z_0) |^2 dxdy + \mathcal{H}^1(\bar K_\delta \cap \{z=z_0\})$$
for all $z_0$ fixed.

Now let us integrate over $z$ and compute 
\begin{eqnarray}
J(u_\delta, P_0) -\delta^2|\mathcal{C}| &\leq& \int_{-1}^1\int_{B_{2D}(0,1)} |\nabla_{x,y} \bar u_{\delta}(z) |^2 dxdydz + \int_{-1}^1\mathcal{H}^1(\bar K_\delta \cap \{z=t\})dt \notag \\
&= &  \int_{\mathcal{C}} |\nabla \bar u_{\delta}|^2 - \int_{\mathcal{C}} \left|\frac{\partial \bar u_\delta}{\partial z}\right| ^2  +  \int_{-1}^1\mathcal{H}^1(\bar K_\delta \cap \{z=t\})dt. \notag \\
&=& J(\bar u_\delta, K_\delta) - \int_{\mathcal{C}} \left|\frac{\partial \bar u_\delta}{\partial z}\right| ^2  +  \int_{-1}^1\mathcal{H}^1(\bar K_\delta \cap \{z=t\})dt - \mathcal{H}^2(K_\delta). \notag \\
&= &  J(u_\delta, P_0) - s(\delta) - \int_{\mathcal{C}} \left|\frac{\partial \bar u_\delta}{\partial z}\right| ^2  +  \int_{-1}^1\mathcal{H}^1(\bar K_\delta \cap \{z=t\})dt - \mathcal{H}^2(K_\delta).\notag
\end{eqnarray}
In other words,
$$s(\delta)\leq \delta^2|\mathcal{C}|- \int_{\mathcal{C}} \left|\frac{\partial \bar u_\delta}{\partial z}\right| ^2  +  \int_{-1}^1\mathcal{H}^1(\bar K_\delta \cap \{z=t\})dt - \mathcal{H}^2(\bar K_\delta).$$
 Next we apply the co-area formula (Th. 2.93 page 101 of \cite{afp}) to obtain
 
 $$\int_{-1}^1\mathcal{H}^1(\bar K_\delta \cap \{z=t\})dt=\int_{\bar K_\delta} c_kd^{\bar K_\delta}f(x) d\mathcal{H}^2(x)$$
 where $c_kd^{\bar K_\delta}f(x) $ is the co-area factor associated to the Lipschitz function 
 $$f:(x,y,z)\to z.$$
 To compute this factor we need to identify $d^{\bar K_\delta}f$.

 For this purpose we let $x\in \bar K_\delta$ be given and $(b_1,b_2)$ an orthonormal  basis of the tangent plane $T_x$ to $\bar K_\delta$ at point $x$. In this basis the linear form $d f$ is represented by the vector $(\langle b_1, {\bf e}_3\rangle,\langle b_2, {\bf e}_3\rangle)$ thus
 $$c_kd^{K_\delta}f(x)=\sqrt{\langle b_1, {\bf e}_3\rangle^2+\langle b_2, {\bf e}_3\rangle^2}=\|P_{T_x}({\bf e}_3)\|,$$
 where $P_{T_x}$ is the orthogonal projection of the vector ${\bf e}_3$ onto the (vectorial) tangent plane $T_x$ at point $x$. In other words,
$$c_kd^{K_\delta}f(x)=\|P_{T_x}({\bf e}_3)\|=|\sin(\theta(x))|$$
  where $\theta(x)$ is the angle between the normal vector $\nu$ to the tangent plane $T_x$ to $\bar K_\delta$ at point $x$, and the vector ${\bf e}_3$. This leads to the   inequality
  
  \begin{eqnarray}
s(\delta)\leq   \int_{\mathcal{C}} \left( \delta^2 - \left|\frac{\partial \bar u_\delta}{\partial z}\right| ^2  \right) dx+  \int_{\bar K_\delta} \left(|\sin(\theta(x))|-1 \right) \;d\mathcal{H}^2(x), \label{inequality0}
 \end{eqnarray} 
 as claimed in the proposition.
 \end{proof}

 Notice that $|\sin(\theta(x))|-1\leq 0$ which means that the above inequality looks stronger than the other one with same flavour just found before (inequality \eqref{ineq00}). We can now give a proof of proposition \ref{propsitionsdelta}.

\begin{proof}[Proof of Proposition \ref{propsitionsdelta}] Since $|\sin(\theta(x))|-1\leq 0$ we deduce  from the above proposition that 

  \begin{eqnarray}
s(\delta)\leq   \int_{\mathcal{C}} \left( \delta^2 - \left|\frac{\partial \bar u_\delta}{\partial z}\right| ^2  \right) dx .\label{inequality01}
 \end{eqnarray} 

On the other hand, for all vertical rays $R_{x,y}$ which do not touch $\bar K_\delta$, i.e. for the $(x,y) \in B_{2D}(0,1)$ such that  the whole segment $(x,y)\times [-1,1] \subset \mathcal{C}\setminus K_\delta$,  we  have
 $$\int_{R_{x,y}}\delta^2- \left|\frac{\partial \bar u_\delta}{\partial z}\right|^2 \leq 0,$$
 because the linear function $z\mapsto \delta z$ is the (1D)-Dirichlet minimizer on this Ray.   
 
 Let us call $Good$ the $(x,y)$ in this situation, and $Bad$ the other ones. Then 
  $$\int_{(x,y) \; Good} \int_{R_{x,y}} \delta^2-\left|\frac{\partial \bar u_\delta}{\partial z}\right|^2 \leq 0$$
  and it follows that the only part which is possibly positive in the right hand side of \eqref{inequality01} is  
 
 $$\int_{(x,y) \; Bad} \int_{R_{x,y}} \delta^2-\left|\frac{\partial \bar u_\delta}{\partial z}\right|^2,$$
 so that 
   \begin{eqnarray}
s(\delta)\leq  \int_{(x,y) \; Bad} \int_{R_{x,y}} \delta^2-\left|\frac{\partial \bar u_\delta}{\partial z}\right|^2 \leq \int_{(x,y) \; Bad} \int_{R_{x,y}} \delta^2.\label{inequality02}
 \end{eqnarray} 
 
Now we invoque the continuity behavior of $(u_\delta,K_\delta)$ as $\delta\to 0$, coming from the theory of Mumford-Shah minimizers (see Proposition 8 Section D.37 page 229 in \cite{d}), which says in particular that    $\bar K_\delta$ converges to $P_0$ as $\delta\to 0$, in Hausdorff distance. Thus there exists some $\varepsilon(\delta)$ going to $0$ as $\delta\to 0$ such that the volume of all bad rays is less than $\varepsilon(\delta)$. This implies
 $$s(\delta)\leq o(\delta^2),$$
 as desired.
 \end{proof}


\section{Numerical results}
\label{lastsection}

We computed some numerical results in a cylinder using the Ambrosio-Tortorelli \cite{at} approximation functional  implemented via the free software \emph{Freefem++}. 
 
In Figure 1 we have represented some isovalues of the field $\varphi$ which represents the phase field approximation of the singular set of a minimizer. We recover the half-plane, which is the unique solution for the   datum $u_0$ on the boundary.

\begin{center}
\includegraphics[width=10cm]{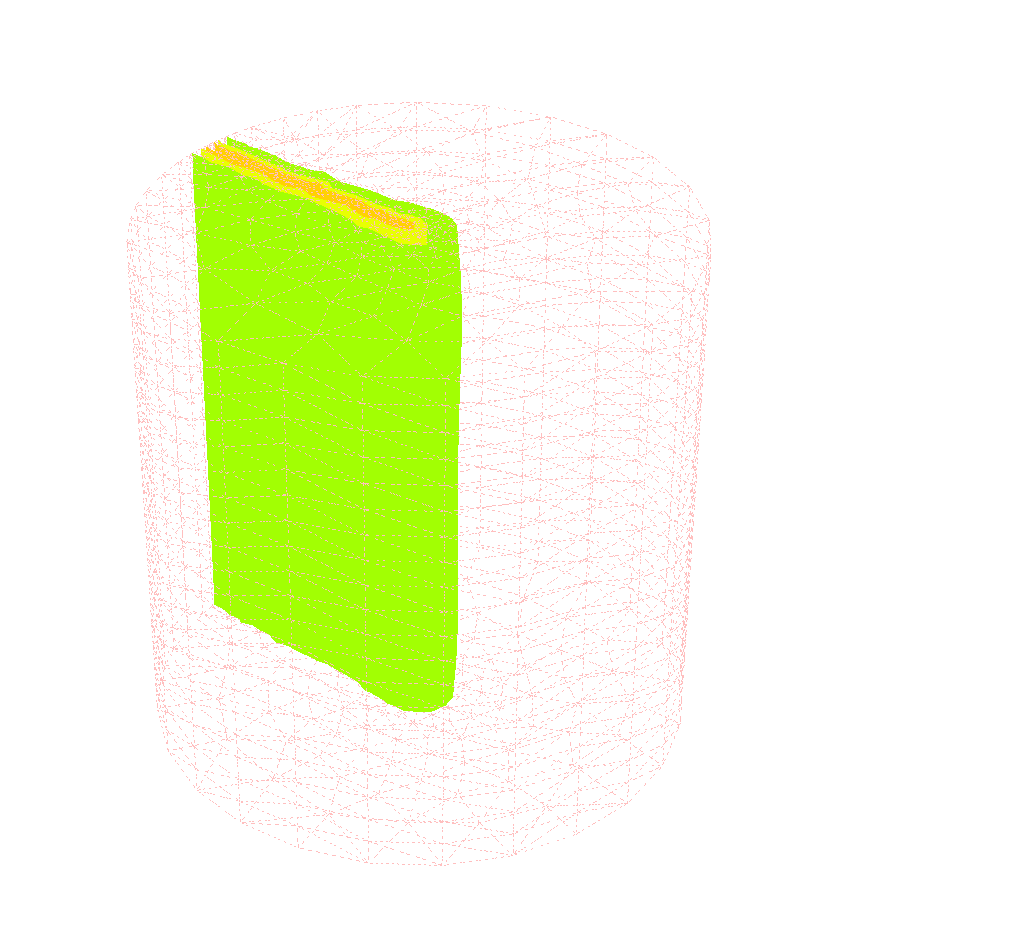}\\
Figure 1: The singular set of the solution with boundary datum $u_0$ on $\partial \mathcal{C}$.
\end{center}

\vspace{1cm}
In Figure 2 we have represented some isovalues of the field $\varphi$  for the datum $u_\delta$ on the boundary of the cylinder. Here $\delta=1/2$. The obtained profile of the solution looks like a twisted half-plane, which indicates that $u_\delta$ may not be a minimizer for this boundary datum. For the recall, the threshold of Proposition \ref{prop1} was $\delta >\sqrt{3-\frac{2}{\pi}}\simeq 1,537$ for which we know for sure that  $u_\delta$ is never minimizing. Here we have tested   $\delta =1/2 < \sqrt{3-\frac{2}{\pi}}$.

\begin{center}
\includegraphics[width=8cm]{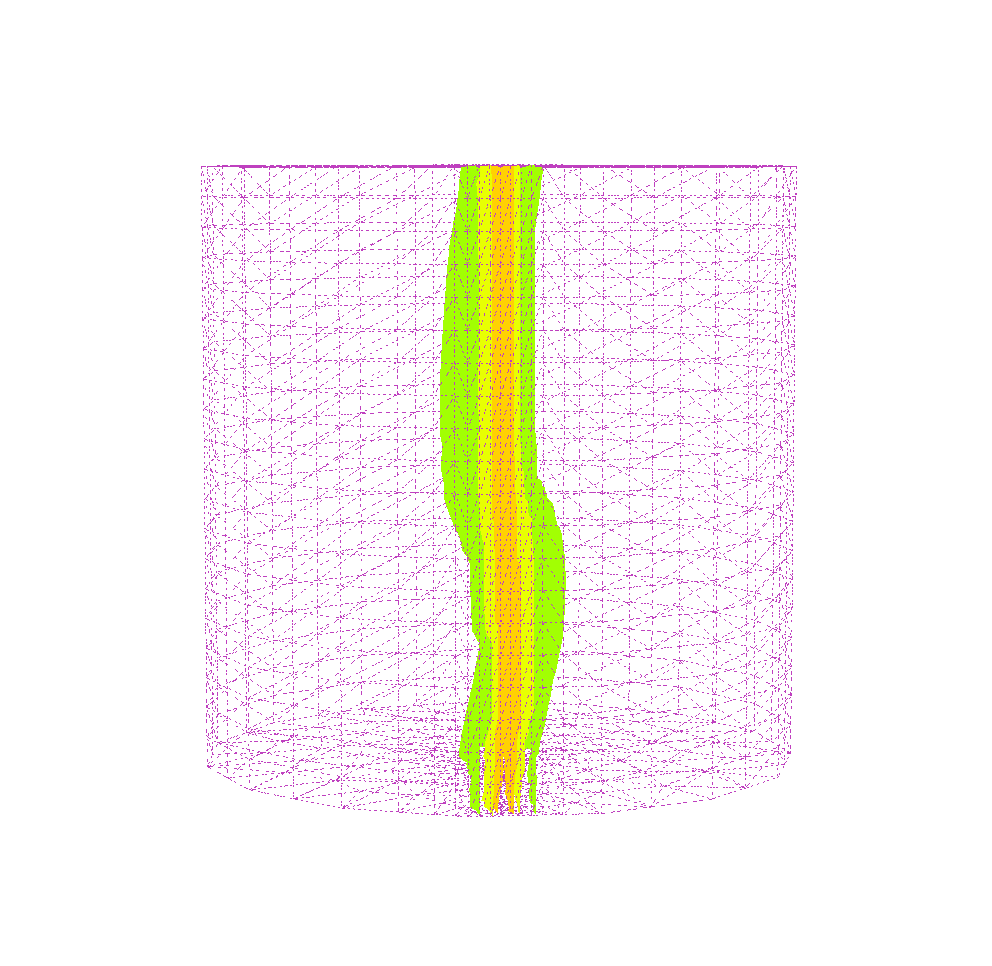}\includegraphics[width=7cm]{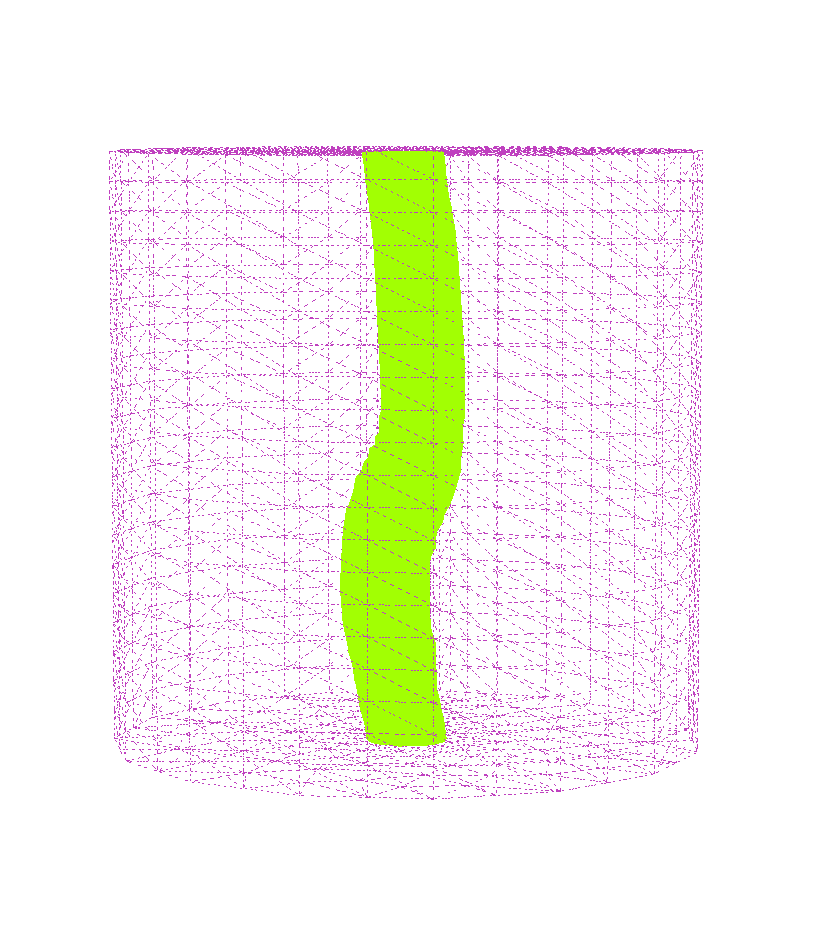}\\

\includegraphics[width=8cm]{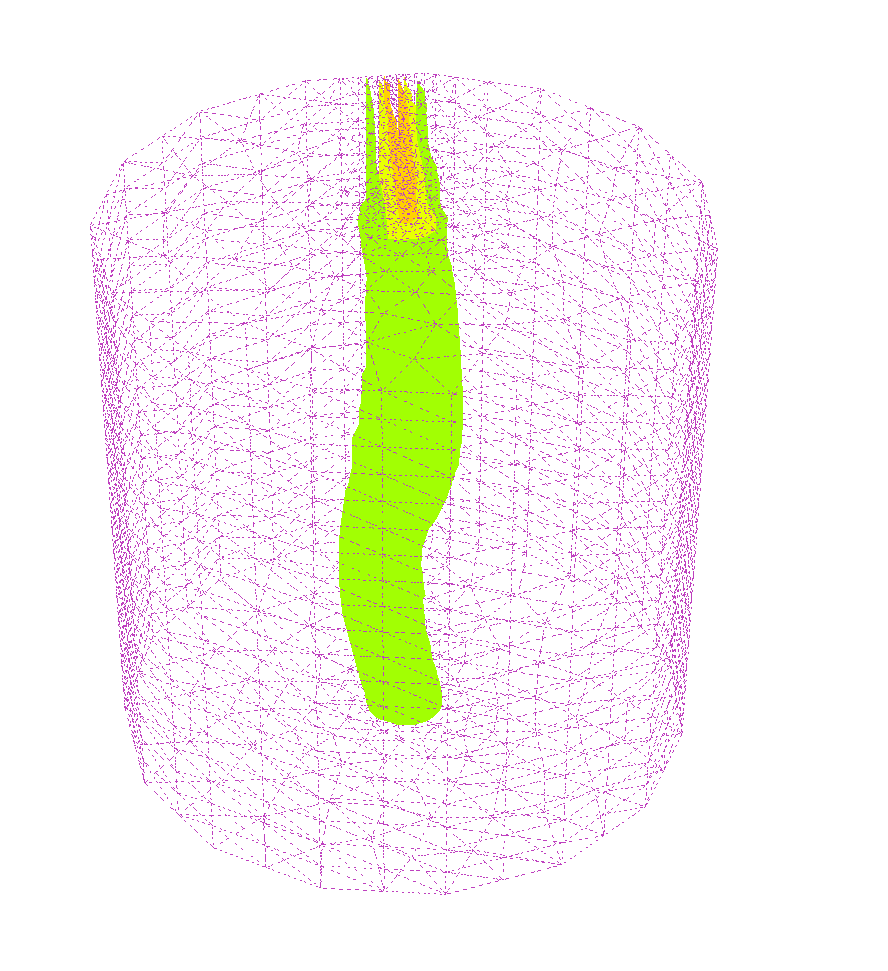}\includegraphics[width=8cm]{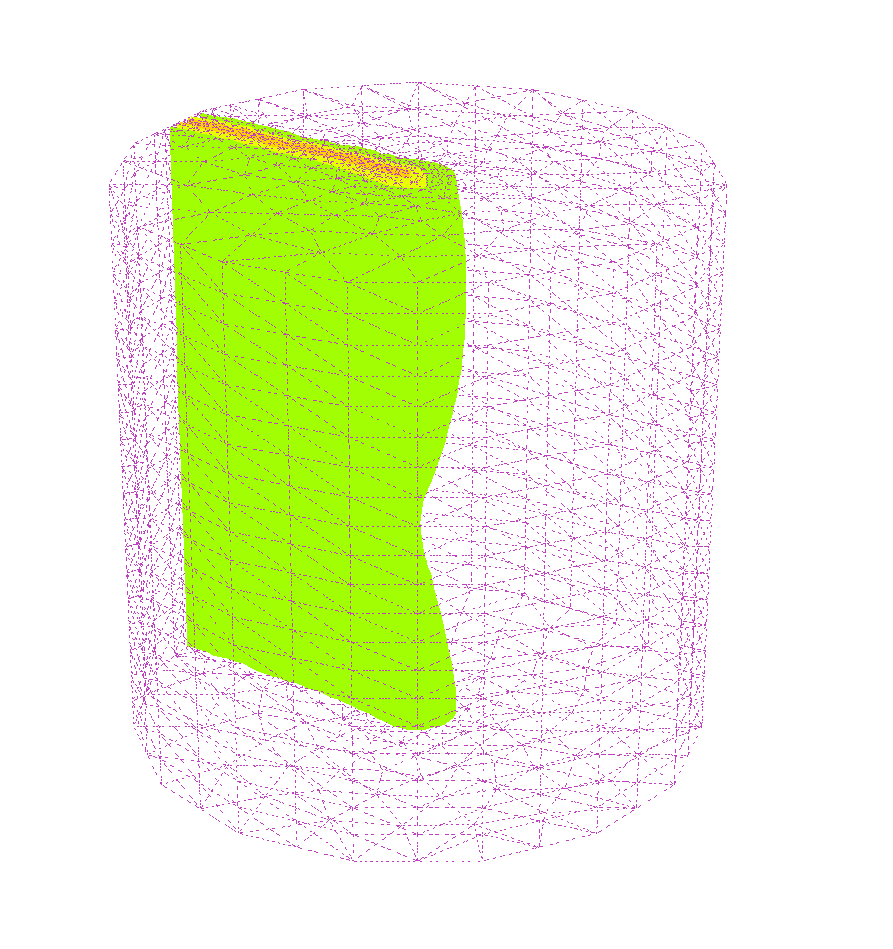}\\
\end{center}

\begin{center}
Figure 2: The singular set of the solution with boundary datum $u_\delta$ on $\partial \mathcal{C}$ with $\delta=1/2$.
\end{center}

\section{Acknowledgements}

We would like to thank Matthieu Bonnivard for giving to us the \emph{Freefem} code implementing the Ambrosio-Tortorelli functional in 3D which was  used in Section \ref{lastsection}. 

This work has been done while the first author was visiting the second author in Ningbo, PR China. Both authors  would like to express their gratitude to the National Science Foundation of China (research grant nr. 11650110437).

The first author has been supported by the PGMO project COCA.


\bibliographystyle{alpha}
\bibliography{biblio}

\vspace{1cm}

Antoine.Lemenant@ljll.univ-paris-diderot.fr  - University Paris 7 - LJLL - FRANCE
 \vspace{0.5cm}

Hayk.Mikayelyan@nottingham.edu.cn  - University of Nottingham Ningbo - CHINA

\end{document}